\documentclass[reqno, 12pt]{amsart}
\pdfoutput=1
\makeatletter
\let\origsection=\section \def\section{\@ifstar{\origsection*}{\mysection}} 
\def\mysection{\@startsection{section}{1}\z@{.7\linespacing\@plus\linespacing}{.5\linespacing}{\normalfont\scshape\centering\S}}
\makeatother        
\usepackage{amsmath,amssymb,amsthm}    
\usepackage{mathabx}\changenotsign    
\usepackage{mathrsfs}
\usepackage{bbm, accents} 

\usepackage{xcolor}  	
\usepackage[backref]{hyperref}
\hypersetup{
	colorlinks,
    linkcolor={red!60!black},
    citecolor={green!60!black},
    urlcolor={blue!60!black},
}

\usepackage{bookmark}

\usepackage[abbrev,msc-links,backrefs]{amsrefs} 
\usepackage{doi}

\renewcommand{\PrintDOI}[1]{\doi{#1}}

\usepackage[T1]{fontenc}
\usepackage{lmodern}

\usepackage[english]{babel}

\numberwithin{equation}{section}

\theoremstyle{plain}
\newtheorem{thm}{Theorem}[section]

\newtheorem{lemma}[thm]{Lemma}

\theoremstyle{definition}
\newtheorem{dfn}[thm]{Definition}
\usepackage{geometry}
\usepackage{microtype}

\def\FF{\mathbbm F}
\def\st{\,|\,}                                 \def\card#1{|#1|}                              
\begin{document}
\title[Parity search]{The parity search problem}
\author{Christian Reiher}
\address{Fachbereich Mathematik, Universit\"at Hamburg,
  D-20146 Hamburg, Germany}
\email{Christian.Reiher@uni-hamburg.de}


\begin{abstract}
We prove that for any positive integers $n$ and $d$ there exists a collection consisting 
of $f=d\log n+O(1)$ subsets $A_1, A_2, \ldots, A_f$ of $[n]$ such that for any two distinct 
subsets $X$ and $Y$ of $[n]$ whose size is at most $d$ there is an index $i\in [f]$ for which 
$\card{A_i\cap X}$ and $\card{A_i\cap Y}$ have different parity. 
Here we think of $d$ as fixed whereas $n$ is thought of as tending to infinity, 
and the base of the logarithm is $2$.

Translated into the language of combinatorial search theory, 
this tells us that 
\[
	d \log n+O(1)
\]
queries suffice to 
identify up to $d$ marked items from a totality of $n$ items 
if the answers one gets are just whether an even or an odd 
number of marked elements has been queried, 
even if the search is performed non-adaptively. 
Since the entropy method easily yields a matching lower bound 
for the adaptive version of this problem, our result is 
asymptotically best possible. 

This answers a question posed by {\sc D\'{a}niel Gerbner} and {\sc Bal\'{a}zs Patk\'{o}s} 
in {\sc Gyula O.H. Katona's} Search Theory Seminar at the R\'{e}nyi institute.
\end{abstract}

\maketitle

\section{Introduction} 

In a typical problem from combinatorial search theory a finite number of 
entities is given to you some of which are considered to be {\it marked} 
or {\it defective} and your task is to find out which of them these are. 
For example, many recreational problems involving coins a few of which 
are forged as well as a scale that may be used to expose them belong to 
this area.  

For a thorough introduction to combinatorial search theory, the reader is 
referred to the excellent and comprehensive survey article \cite{Kat}.

Recently {\sc Gerbner} and {\sc Patk\'{o}s}~\cite{Pat} started to consider 
the following search problem: 
One gets confronted with $n$ items -- the set of which may for convenience 
be identified with the set $[n]=\{1, 2, \ldots, n\}$ -- and one knows in 
advance that at most $d$ of these items are marked, where $0\le d\le n$. 
To identify them, one may make a sequence of queries, i.e., specify a sequence 
of subsets of $[n]$, and each time one makes such a query one is told the
parity of the number of marked elements in ones query set. 
So, for instance, by querying a set containing only one element one learns 
whether this element is marked or not and hence the problem is solvable in 
principle. As usual, however, it is more interesting to think about the 
least number of queries one needs to perform this task. 
More precisely, {\sc Gerbner} and {\sc Patk\'{o}s} asked what the 
asymptotic behaviour of this minimal number is if $d$ is fixed whilst $n$ 
tends to infinity. 

Now actually there are two version of this question. 
In the first of these, called the {\it adaptive problem}, 
one allows ones query sets to depend on the answers one has gotten 
to all previous queries. 
In the second, {\it non-adaptive} version, one has to specify all query 
sets before getting the first answer. The main result of this article asserts 
that for both of these versions $d\log n+O(1)$ queries are necessary and sufficient.    

It is quite standard to obtain a lower bound of the form $d\log n+O(1)$ 
to the adaptive problem, where the base of the logarithm is $2$. 
For if 
\[
	2^m<\sum\limits_{i=0}^{d}\binom{n}{i}=\Theta(n^d)\,,
\]
then it may happen 
that each of the first $m$ answers reduces the number of outcomes still possible 
by no more than a factor of two, for which reason $m$ queries cannot be enough. 
It is also clear that the non-adaptive problem requires no less queries than the 
adaptive one for one may pretend to search adaptively while in fact not caring 
about the answers. 
Thus it suffices to prove an upper bound of the form $d\log n+O(1)$ to the 
non-adaptive problem.     

It seems worth while to observe that the non-adaptive problem may also be 
viewed as a question from extremal set theory. Specifically, one is interested 
in the number $f(n, d)$ defined as follows:

\begin{dfn}
Given two positive integers $n$ and $d$, let $f(n, d)$ be the least integer $f$ 
such that there exist $f$ subsets $A_1, A_2, \ldots, A_f$ of $[n]$ with the 
following property: For any two distinct subsets $X$ and $Y$ of $[n]$ the size 
of which is at most $d$, there exists an index $i\in [f]$ such that the 
cardinalities of $A_i\cap X$ and $A_i\cap Y$ have different parity.
\end{dfn}

I would like to record here that {\sc Gerbner} and {\sc Patk\'{o}s} showed 
that choosing these sets~$A_i$ uniformly at random one can get 
$f(n, d)\le 2d\log n+O(1)$. Their proof uses the first moment method. 
It may be observed that a routine application of the symmetric version 
of {\sc Lov\'{a}sz's} Local Lemma (see \cite{ErdLov} or Corollary 5.1.2 
from \cite{AlonSpen}) would allow us to improve this to $f(n, d)\le (2d-1)\log n+O(1)$. 
But in fact we shall prove $f(n, d)=d\log n+O(1)$ below. Somewhat more explicitly, 
we shall get

\begin{thm} \label{thm:main}
If $d$, $m$, and $n$ denote three positive integers with ${dm\le n<2^m}$, 
then $f(n, d)\le dm$.
\end{thm} 

The proof will be given in the next section.
  
\section{The proof of Theorem~\ref{thm:main}}

The actual proof of Theorem~\ref{thm:main} is prepared by a sequence of three 
lemmata most of which are of an algebraic nature. 
Throughout we denote the finite field with $q$ elements by 
$\FF_q$ and refer to the multiplicative group of its nonzero elements by 
$\FF_q^\times$. If $F$ is a field we write $F^n$ for the $n$-dimensional 
vector space over $F$. Finally we would like to remind the reader that the number 
of ones appearing in a vector from $\FF_2^{n}$ is sometimes called its {\it weight}.

The basic strategy of our proof is as follows: one interprets the problem as a statement 
about vector spaces over $\FF_2$ and applies a change of basis to see that all 
one needs to do is proving Lemma~\ref{lem:23}. 
Using a direct sum decomposition this task can be reduced to showing Lemma~\ref{lem:22}, 
which in turn is accomplished by means of an explicit construction based on the 
following algebraic fact exploiting the multiplicative structure of fields having 
characteristic~$2$.
     
\begin{lemma}\label{lem:21}
If $A\not\subseteq\{0\}$ denotes a finite subset of a field of characteristic $2$, 
then for some odd positive integer $k\le\card{A}$ one has 
$\sum\limits_{x\in A}x^k\ne 0$.   
\end{lemma}

\begin{proof}
Pick any $a\in A-\{0\}$ and observe that
\[
	\sum_{x\in A}x\prod_{b\in A-\{a\}}(x-b)=a\prod_{b\in A-\{a\}}(a-b)\ne 0\,.
\]
Expanding the product appearing under the sum of the left hand side and rearranging we get 
\[
	\sum_{i=1}^{\card{A}}\alpha_i\sum_{x\in A}x^i\ne 0
\]
with certain irrelevant coefficients $\alpha_1, \ldots, \alpha_{\card{A}}$ 
from our base field.  
Thus there exists some positive integer $k'\le\card{A}$ such that 
$\sum\limits_{x\in A}x^{k'}\ne 0$. Now if $k$ denotes the least such $k'$, 
then $k$ automatically has to be odd, for otherwise we could use the equation 
\[
	\sum_{x\in A}x^{k}=\Bigl(\sum_{x\in A}x^{k/2}\Bigr)^2
\]
to obtain a contradiction. Thereby our lemma is proved. 
\end{proof}

\begin{lemma}\label{lem:22}
For any two positive integers $d$ and $m$, the $\FF_2$-vector space 
$\FF_2^{dm}$ has a generating subset $B$ of size at least $2^m-1$ 
such that each vector admits at most one representation as the sum of at most 
$d$ distinct members of $B$.
\end{lemma}

\begin{proof}
Plainly it suffices to exhibit a set consisting of $2^m-1$ vectors from 
$\FF_2^{dm}$ possessing the unique representability property. 
For once we have found such a set $B$, we may look at a direct sum decomposition 
$\FF_2^{dm}=\langle B\rangle\oplus U$ with some vector space $U$ 
and extend $B$ by a basis of $U$ to achieve both goals. 

For the purpose of finding such a set $B$, we may evidently replace the 
vector space $\FF_2^{dm}$ appearing in this statement by $\FF_{2^m}^d$. 
Corresponding to each number $\xi\in\FF_{2^m}$ we define $v_\xi$ 
to be the vector $(\xi, \xi^3, \ldots, \xi^{2d-1})$ from the latter space, 
and then we claim that
\[
	B=\{v_\xi\st \xi\in\FF_{2^m}^{\times} \}
\]
is as desired. To see this, suppose that some vector $x$ admitted 
two distinct representations as the sum of at most $d$ elements 
from $B$. Adding these representations up and canceling terms appearing twice, 
we obtain a nonempty subset $A$ of $\FF_{2^m}^{\times}$ 
whose size is at most $2d$ such that $\sum\limits_{\xi\in A}v_\xi=0$. 
So in particular for all odd $k\le\card{A}$ we have $\sum\limits_{\xi\in A}\xi^k=0$, 
contrary to our previous lemma.
\end{proof}

\begin{lemma}\label{lem:23}
Let $d$, $m$, and $n$ denote three positive integers such that $dm\le n<2^m$. 
Then there is some vector subspace of $\FF_2^{n}$ of codimension $dm$ 
containing no nonzero vector whose weight is at most $2d$.
\end{lemma}

\begin{proof}
It is convenient to think of $\FF_2^{n}$ as being the space 
$V=\FF_2^{dm}\oplus\FF_2^{n-dm}$. 
It has $W=\{0\}\oplus\FF_2^{n-dm}$ as a subspace of codimension $dm$. 
By our foregoing lemma there exist $n$ distinct vectors $b_1, b_2, \ldots, b_n$ 
from $\FF_2^{dm}$ such that each vector from this space is expressible
in at most one way as the sum of at most $d$ distinct vectors from this sequence, 
and such that $b_1, b_2, \ldots, b_{dm}$ is a basis of this space. 
Now define $v_1, v_2, \ldots, v_{dm}$ to be the zero vector of $\FF_2^{n-dm}$ 
and let $v_{dm+1}, v_{dm+2}, \ldots, v_{n}$ be any basis of this space. 
Clearly the set 
\[
	L=\{(b_i, v_i)\st i=1, 2, \ldots, n\}
\]
forms a basis of $V$ and by our construction it is not possible that a 
nonempty sum comprised of at most $2d$ distinct terms from $L$ gives a 
vector from $W$. 
Any automorphism of $V$ sending $L$ to the standard basis 
maps $W$ onto a vector subspace $W'$ whose 
codimension is still $dm$ and that does not contain any nonzero vector whose 
weight is at most $2d$. So $W'$ is as desired. 
\end{proof} 

We are now ready to prove Theorem~\ref{thm:main}. To do so we identify the power set 
of $[n]$ with the vector space $\FF_2^{n}$ via characteristic functions. 
It is well known that the parity of the size of the intersection of two sets 
thus corresponds to the standard scalar product. 
Our task now consists in exhibiting $dm$ vectors $v_1, v_2, \ldots, v_{dm}$ 
such that for any two distinct vectors~$x$ and $y$ the weight of which is at 
most $d$ there is some $i\in[dm]$ with $x\cdot v_i\ne y\cdot v_i$. 
This may be achieved by taking $W$ to be a vector subspace of $\FF_2^{n}$ 
as obtained in our third lemma and then choosing the vectors $v_1, v_2, \ldots, v_{dm}$ 
so as to span its orthogonal complement. 
Given any two distinct vectors $x$ and $y$ from $\FF_2^{n}$ whose weights 
are at most $d$, one easily sees that their difference is nonzero and has weight 
at most $2d$. Therefore it cannot belong to $W$, which in turn means that there 
is indeed some $i\in[dm]$ satisfying $(x-y)\cdot v_i\ne 0$. 
This completes the proof of our main result, Theorem~\ref{thm:main}.     

\subsection*{Acknowledgement} I would heartily like to thank {\sc Gyula O.H. Katona} 
for inviting me to the Alfred R\'{e}nyi Institute for one week, where this piece of research 
was carried out, and {\sc D\'{a}niel Gerbner} and {\sc Bal\'{a}zs Patk\'{o}s} 
for telling me about the problem solved here.

\end{document}